\documentclass{amsart}

\usepackage{amssymb}
\usepackage{mathrsfs}
\usepackage{stmaryrd}
\usepackage[colorlinks = true]{hyperref}
\usepackage{mathabx}

\usepackage{tikz}
\usetikzlibrary{cd}

\usepackage{biblatex}
\usepackage{stackrel}

\bibliography{lectures}

\newcommand{\bk}{\Bbbk}
\newcommand{\Z}{\mathbb{Z}}
\newcommand{\F}{\mathbb{F}}

\newcommand{\scH}{\mathscr{H}}
\newcommand{\scL}{\mathscr{L}}

\newcommand{\Spec}{\mathrm{Spec}}

\newcommand{\ba}{\mathbf{a}}
\newcommand{\bb}{\mathbf{b}}
\newcommand{\bc}{\mathbf{c}}
\newcommand{\bi}{\mathbf{i}}
\newcommand{\bj}{\mathbf{j}}

\newcommand{\Gr}{\mathrm{Gr}}

\makeatletter
\def\lotimes{\@ifnextchar_{\@lotimessub}{\@lotimesnosub}}
\def\@lotimessub_#1{\mathchoice{\mathbin{\mathop{\otimes}^L}_{#1}}%
  {\otimes^L_{#1}}{\otimes^L_{#1}}{\otimes^L_{#1}}}
\def\@lotimesnosub{\mathbin{\mathop{\otimes}^L}}
\makeatother

\numberwithin{equation}{section}

\theoremstyle{definition}

\theoremstyle{remark}

\newtheorem{theorem}{Theorem}
\newtheorem{proposition}[theorem]{Proposition}
\newtheorem{definition}[theorem]{Definition}
\newtheorem{corollary}[theorem]{Corollary}
\newtheorem{conjecture}[theorem]{Conjecture}

\newtheorem{lemma}[theorem]{Lemma}

\theoremstyle{remark}

\newcommand{\A}{\mathbb{A}}

\newcommand{\G}{\mathbb{G}}

\newcommand{\p}{\mathbb{P}}
\newcommand{\Zhat}{\widehat{\mathbb{Z}}}

\newcommand{\FinS}{\mathrm{FinS}}
\newcommand{\FinSs}{\mathrm{FinS}^{\mathrm{sur}}}

\newcommand{\Gal}{\mathrm{Gal}}

\newcommand{\atimes}{\stackbin{\leftarrow}{\times}}

\newcommand{\et}{\text{\'et}}
\newcommand{\proet}{\text{pro\'et}}

\newcommand{\qlbar}{\overline{\mathbb{Q}}_l}

\title{Unipotent nearby cycles and nearby cycles over general bases}
\author{Andrew Salmon}
\email{asalmon@alum.mit.edu}

\begin{document}

\maketitle

\begin{abstract}
    We show that under some conditions, two constructions of nearby cycles over general bases coincide.  More specifically, we show that under the assumption of $\Psi$-factorizability, the constructions of unipotent nearby cycles over an affine space in~\cite{achar2023higher} can be described using the theory of nearby cycles over general bases via the vanishing topos.  In particular, this applies to nearby cycles of Satake sheaves on Beilinson-Drinfeld Grassmannians with parahoric ramification.
\end{abstract}

\section{Introduction}

Let $S$ and $X$ be schemes of finite type over an algebraically closed base field $\mathbb{F}$ equipped with a map $f \colon X \to S$.  Over $X$, we may define $D^b_c(X, \bk)$, the bounded derived category of constructible sheaves with $\bk$ coefficients.  We will mostly restrict to the case that $\bk$ is a finite extension of $\mathbb{Q}_{\ell}$ or $\mathbb{F}_{\ell}$, with $\ell$ invertible in $\mathbb{F}$, or an algebraic closure thereof, where the corresponding derived categories of sheaves can be defined via the adic formalism.  If $S$ is one-dimensional with generic point $\eta$ and special point $s$, the classical theory of nearby cycles gives a functor between derived categories of constructible sheaves from a generic fiber to a special fiber of $X$ over $S$.  In particular, if $\eta$ is a generic point of $S$ and $s$ is a special point, let $X_{\eta}$, resp. $X_s$ be the fibers of $X$ over these points.  There is a nearby cycles functor $\Psi \colon D^b_c(X_\eta, \bk) \rightarrow D^b_c(X_s, \bk)$, which is t-exact with respect to the perverse t-structure and admits an action of the local inertia group.  In~\cite{beilinson1987glue}, Beilinson gave a construction of unipotent nearby cycles, where the action of local inertia factors through tame inertia and the tame inertia acts unipotently.

For some applications, it is desirable to have a version of nearby cycles over a base $S$ that is not one-dimensional.  Two such approaches exist that supply such a theory.  The first is a generalization of Beilinson's construction to an arbitrary base, introduced in~\cite{gaitsgory2004braiding} and subsequently generalized in~\cite[Section~4]{salmon2023unipotent},~\cite[Section~4.6]{eteve2023monodromiques}, and~\cite{achar2023higher}.  When defined, for a perverse sheaf $K$ on a scheme $X$ over $\A^n$, this produces a perverse sheaf $\Upsilon(K)$ on the special fiber of $X$ over $0 \in \A^n$.  More generally, Achar and Riche show that for any map $\alpha \colon P_* \rightarrow Q_*$ of pointed finite sets, there is a (conditionally defined) perverse sheaf on $X \times_{\A^P \times \bar{0}} \A^Q \times \bar{0}$ that is written $\Upsilon^{\alpha}(K)$.  The second approach is the theory of nearby cycles over general bases using the vanishing topos, as in~\cite{orgogozo2006modifications}.  This approach works over a general base $S$, not just affine space, and the study of monodromy of these nearby cycles is more complicated.  For any \'etale sheaf $K$ on $X$ and any specialization $s \rightarrow t$ of geometric points on the etale topos of the base $S$, there is an object $R(\Psi)_{s}^{t} K$ in the derived category of \'etale sheaves.

In Corollary~\ref{cor: unipotent general bases} and Corollary~\ref{cor: unipotent general bases ula}, we give conditions under which $\Upsilon^{\alpha}(K)$ can be understood as coming from $R(\Psi)_{s}^{t} K$.  In particular, setting $S = \A^P$, surjective maps $\alpha \colon P_* \rightarrow Q_*$ as above allows us to choose a specialization $F(\alpha)$ in the \'etale topos.  Assuming that $(f, K)$ satisfy a condition called $\Psi$-factorizability, then $(\Psi_f)_{F(\alpha)} K[|Q|-|P|]$ is perverse, and there is a canonical subobject
\[(\Psi^u_f)_{F(\alpha)} K[|Q|-|P|] \subseteq (\Psi_f)_{F(\alpha)} K[|Q|-|P|].\]
Under an additional assumption that $(f, K)$ is universally locally acyclic when restricted to $\G_m^P \subset \A^P$, Corollary~\ref{cor: unipotent general bases ula} gives an isomorphism
\[ \Upsilon^{\alpha}_f(K) \cong (\Psi^u_f)_{F(\alpha)} K[|Q|-|P|]. \]

To motivate the notion of $\Psi$-factorizability, in geometric representation theory, one might consider a one-parameter family, the Beilinson-Drinfeld Grassmannian ramified at a closed point $x$ on a smooth curve $C$.  Satake sheaves naturally exist away from the closed point which categorify elements of a spherical Hecke algebra, and their nearby cycles at the closed point give a categorification of central elements in the affine Hecke algebra, realizing the Bernstein isomorphism \cite{gaitsgory1999construction}.  One hopes, then, to show that the nearby cycles functor is a central functor from a tensor category to the monoidal affine Hecke category.  Checking that this functor satisfies braiding relations, expected of a central functor, requires the use of a two-parameter degeneration of a Beilinson-Drinfeld Grassmannian over a two-dimensional base, and one must show that nearby cycles commute with one another \cite{gaitsgory2004braiding}.  Since there is no natural map from $\Psi_1 \Psi_2$ to $\Psi_2 \Psi_1$, one must really construct an object that naturally maps to both, which is a nearby cycle over a two-dimensional base, in which case these maps are naturally isomorphisms.  $\Psi$-factorizability generalizes this situation where we require that maps to iterated nearby cycles are an isomorphism, and its definition requires the existence of an isomorphism~\ref{eqn: psi factorizability definition} for any composition of specializations of geometric points along the base.

There is an interesting analogy to be made to the factorization structure that Satake sheaves exhibit, whereby specialization to the diagonal on the Beilinson-Drinfeld Grassmannian, that is, the fusion product, is the same as the convolution product.  The centrality of the construction of nearby cycles sheaves in the affine Hecke category \cite{gaitsgory2004braiding} is a ramified version of this property and can be viewed as a generalization of this property in light of the result of Richarz that the Satake sheaves over the Beilinson-Drinfeld Grassmannian are universally locally acyclic \cite{richarz2014new}.  That is, under the assumption of local acyclicity, taking nearby cycles is the same as restriction to the fiber, and so commuting nearby cycles would then encode commutativity of the convolution product in the Satake category.  We note that Satake sheaves on an iterated Beilinson-Drinfeld Grassmannian over $\A^n$ with parahoric ramification at $0 \in \A^1$ satisfy the assumptions of Corollary~\ref{cor: unipotent general bases ula}.  In fact, Corollary~\ref{cor: maps agree} shows that certain ``fusion'' maps constructed in this setting also agree.  Such fusion isomorphisms have found use in the Langlands program, where they play a role in showing cases where the shriek pushforward of IC sheaves on the moduli space of shtukas, which is not proper over the base, nevertheless commutes with nearby cycles \cite[Theorem~4.12]{salmon2023restricted}.

\subsection*{Acknowledgements}

This paper was originally written as an appendix to \cite{achar2023higher}, and I would like to thank Pramod Achar and Simon Riche for their interest in this work.  I would especially like to thank Simon Riche for careful readings of drafts of this paper and for many comments that improved the paper's correctness, clarity, and quality.  I would also like to thank an anonymous referee identifying several points where the paper had gaps or needed to be improved.

\section{Nearby cycles over general bases}

\subsection{Review of nearby cycles over general bases}

Consider morphism $f \colon X \rightarrow S$ between schemes of finite type over an algebraically closed field $\mathbb{F}$, two geometric points with a specialization $t \rightarrow s$ on the base $S$ that gives a morphism $t \rightarrow S_{(s)}$ to the \'etale localization, and a bounded-below complex of sheaves $K \in D^+(X_{t}, \bk)$.  In this context, we will define nearby cycles functors.
\begin{definition}
    The shredded nearby cycle $R(\Psi_f)_{t}^s \colon D^+(X_{t}, \bk) \rightarrow D^+(X_s, \bk)$ is the functor $i^* Rj_*$, where $j \colon X_{t} \rightarrow X_{(s)}$ and $i \colon X_s \rightarrow X_{(s)}$.
\end{definition}
To simplify notation, we will freely drop the structure map and write $R(\Psi)_{t}^s$ when $f$ is clear.  Let $x$ be a geometric point of $X$ such that $f(x) = s$.  Recall that the Milnor tube is the scheme $X_{(x)} \times_{S_{(s)}} S_{(t)}$.  In this setting, there is a natural identification between the stalks of nearby cycles and the cohomology of the Milnor tube.  That is,
\begin{equation} 
(R(\Psi_f)_{t}^{s} K)_x \cong R\Gamma(X_{(x)} \times_{S_{(s)}} S_{(t)}, K|_{X_{(x)} \times_{S_{(s)}} S_{(t)}}).
\end{equation}

We must make a comment about the meaning of derived categories such as $D^b(X, \bk)$ here.  If $\bk$ is torsion, we may use $D^b(X, \bk)$ to denote the bounded derived category of \'etale sheaves.  In the case when $\bk$ is a finite extension of $\mathbb{Q}_{\ell}$ or an algebraic closure thereof, instead of the \'etale topology, we really want a topos containing the constructible $\ell$-adic sheaves, thereby allowing for limiting constructions.  Such a topos has been constructed by Bhatt and Scholze and is known as the pro-\'etale topology \cite{bhatt2013pro}.  Thus, we let $D(X, \bk)$ be the derived category of all pro-\'etale sheaves.  Inside this category there is a bounded below subcategory of pro-\'etale sheaves $D^+(X, \bk)$, a bounded derived category $D^b(X, \bk)$, and a bounded derived category of constructible sheaves, which we denote $D^b_c(X, \bk)$.  We may also consider the category of abelian sheaves on a topos, which we denote $Ab(X)$ and its derived category $D(X)$.  When $X$ is a point (the topos $\mathrm{Set}$) and $R$ is a ring, we denote the derived category of sheaves of $R$-modules $D(X, R)$ as just $D(R)$.

When we need to make the distinction between \'etale and pro-\'etale categories, we write $X_{\et}$ for the \'etale topos and $X_{\proet}$ for the pro-\'etale topos.  One preliminary result is that the shredded nearby cycles functor has finite cohomological dimension.

\begin{proposition}\label{prop: finite cd}
    For a map $f \colon X \rightarrow S$ be locally of finite type whose fibers have dimension at most $N$.  Let $K$ be an \'etale, resp.\ pro-\'etale sheaf of $\bk$-modules.  Then $R^i(\Psi)_{t}^{s} K = 0$ for $i > 2N$.
\end{proposition}

\begin{proof}
The finite cohomological dimension of $\Psi$ is due to Gabber in the case of torsion coefficients in the \'etale topology \cite[Proposition~3.1]{orgogozo2006modifications} by reduction to the Milnor tube.  The proof follows word for word in the pro-\'etale setting.
\end{proof}

Gabber's proof in \cite[Proposition~3.1]{orgogozo2006modifications} suggests a technique to transfer results from the \'etale topology to the pro-\'etale topology, and we follow the strategy in \cite[Section~6.5]{bhatt2013pro} for direct image functors.  Let $R$ be the ring of integers in $\bk$, temporarily assumed to be a finite extension of $\mathbb{Q}_{\ell}$, and let $\mathfrak{m}$ be the maximal ideal so that $R = \lim R / \mathfrak{m}^n$.
For $M \in D(R)$, write $\underline{M}$ as its pullback to a sheaf on $D(X, R)$.  For a sheaf $K \in D^b_c(X_{\proet}, R)$, we may write $K \otimes_R \underline{M} = R\lim (K_n \otimes_{R/\mathfrak{m}^n} \underline{M_n})$.  We form a map
\[
    R(\Psi)_{t}^{s}(K \otimes_R \underline{M}) \rightarrow R(\Psi)_{t}^{s}(K) \otimes_R \underline{M}
\]
the adjoint of the composite map
\[
\begin{aligned}
j_* R\lim(K_n \otimes_{R/\mathfrak{m}^n} \underline{M_n}) &\simeq R\lim j_* (K_n \otimes_{R/\mathfrak{m}^n} \underline{M_n}) \\
&\rightarrow R\lim i_* i^* j_* (K_n \otimes_{R/\mathfrak{m}^n} \underline{M_n}) \\
&\simeq i_* R\lim i^* j_* (K_n \otimes_{R/\mathfrak{m}^n} \underline{M_n}).
\end{aligned}
\]

The following is an analogue of \cite[Lemma~6.5.11.(ii)]{bhatt2013pro}.

\begin{proposition}\label{prop: inverse limit commutes}
    Let $K \in D^b_c(X_{\proet}, R)$ and $M \in D^-(R)$.  The map
    \begin{equation}
        R(\Psi)_{t}^{s}(K \otimes_R \underline{M}) \rightarrow R(\Psi)_{t}^{s}(K) \otimes_R \underline{M}
    \end{equation}
    is an isomorphism.
\end{proposition}

\begin{proof}
    It suffices to check this property on stalks, i.e.\ on the cohomology of the Milnor tube, where it becomes the map
    \begin{equation}
        R\Gamma(X_{(x)} \times_{S_{(s)}} S_{(t)}, K \otimes_R \underline{M}) \rightarrow R\lim R\Gamma(X_{(x)} \times_{S_{(s)}} S_{(t)}, K_n \otimes_{R/\mathfrak{m}^n} \underline{M_n}).
    \end{equation}
    Write $\tau \colon X_{(x)} \times_{S_{(s)}} S_{(t)} \rightarrow X_{(s)} \times_{S_{(s)}} S_{(t)}$.  Then for a sheaf $L$, we have
    \[ R\Gamma(X_{(x)} \times_{S_{(s)}} S_{(t)}, L) \simeq R\Gamma(X_{(s)} \times_{S_{(s)}} S_{(t)}, R\tau_* L) \simeq R\Gamma(X_t, (R\tau_* L)|_{X_t}) \]
    where the last isomorphism follows by proper base change \cite[Lemma~6.7.5]{bhatt2013pro}.  Now, $\tau$ has cohomological dimension $0$ and global sections of $X_t$ has cohomological dimension at most $2N$, so \cite[Lemma~6.5.11.(ii)]{bhatt2013pro} implies that the map
    \[
    R\Gamma(X_t, (R\tau_* R\lim (K_n \otimes_{R/\mathfrak{m}^n} \underline{M_n}))|_{X_t}) \rightarrow R\lim R\Gamma(X_t, (R\tau_* K_n \otimes_{R/\mathfrak{m}^n} \underline{M_n})|_{X_t})
    \]
    is an isomorphism.
\end{proof}

The most basic pathology of the definition of shredded nearby cycles is the failure of constructibility.  For example, let $S = \p^2$ over an algebraically closed base field, and let $X$ be the blow-up of $S$ at a point $0 \in \p^2$.  Let $K$ be the constant sheaf $\mathbb{Z} / n$ on $X$, and take our specialization $t \rightarrow s$ to be a specialization from a geometric point over the (two-dimensional) generic point of $S$ to the point $0$.  Then $R(\Psi_f)_t^s K$ is not constructible, as shown by Orgogozo \cite[Section~11]{orgogozo2006modifications}.  To correct this situation, the theory of nearby cycles over general bases considers the topoi $X$ and $S$ and uses the vanishing topos $X \atimes_S S$, which is a 2-categorical construction for topoi \cite[Expos\'e~11]{illusie2014travaux}.  The vanishing topos is equipped with a geometric morphism $\Psi \colon X \rightarrow X \atimes_S S$, and we may consider nearby cycles as pushforward along this geometric morphism $R\Psi_f = (\Psi_f)_*$.  The shredded nearby cycle above is then recovered by pulling back along $X_{s} \times_s t \rightarrow X \atimes_S S$.

For the vanishing topos, there is a natural base change property with respect to maps $S' \rightarrow S$ which leads to the concept of a pair $(f, K)$ being $\Psi$-good.  For $h \colon S' \rightarrow S$, let $X' = X \times_S S'$, $h' \colon X' \rightarrow X$, and $f' \colon X' \rightarrow S'$ be defined by pullback.  We may form the diagram \cite[Equation~1.4.2]{illusie2017around}
\[
    \begin{tikzcd}
        X' \arrow{r}{h'} \arrow{d}{\Psi_{f'}} & X \arrow{d}{\Psi_f} \\
        X' \atimes_{S'} S' \arrow{r}{h' \atimes_h h} & X \atimes_S S
    \end{tikzcd}
\]
together with a base change morphism $(h' \atimes_h h)^* R\Psi_f \rightarrow R\Psi_{f'} (h')^*$.
\begin{definition}
    We say that the pair $(f, K)$ is $\Psi$-good if for any map $h$, the base change morphism is an isomorphism.
\end{definition}

We remark that $\Psi$-goodness is automatic for a 1-dimensional base, either a smooth curve or a henselian discrete valuation ring.  We will use various properties of $\Psi$-goodness which have been proven for \'etale sheaves.

\begin{proposition}\label{prop: constructible}
Let $(f,K)$ be $\Psi$-good and let $K \in D^b_c(X,\bk)$ be constructible.  Let $\alpha \colon t \to s$ be a specialization in $S$.  Then $R(\Psi_f)_{t}^{s} K \in D^b_c(X_{s},\bk)$.
\end{proposition}

\begin{proof}
The constructibility is essentially due to Orgogozo and appears as \cite[Theorem~8.1]{orgogozo2006modifications} in the torsion coefficients case.

These results may be transferred to the pro-\'etale setting.  Constructibility for a complex $K$ of $R$-modules on $X_{\proet}$ as in \cite[Definition~6.5.1]{bhatt2013pro} involves an $\mathfrak{m}$-adic completeness and the property that $K \otimes R / \mathfrak{m}$ is obtained by pulling back along the geometric morphism $X_{\proet} \rightarrow X_{\et}$.  Completeness follows by identifying the inverse systems $\{ R(\Psi)_{t}^{s} (K) \otimes_R R / \mathfrak{m}^n \} \simeq \{ R(\Psi)_{t}^{s} (K \otimes_R R / \mathfrak{m}^n) \}$.  By applying Proposition~\ref{prop: inverse limit commutes} with $M = R / \mathfrak{m}$, $R(\Psi)_{t}^{s} (K \otimes_R R / \mathfrak{m}) = R(\Psi)_{t}^{s}(K) \otimes_R R / \mathfrak{m}$ is constructible by Orgogozo's theorem.  We conclude that $R(\Psi)_{t}^{s}$ preserves constructible pro-\'etale complexes analogous to \cite[Lemma~6.7.2]{bhatt2013pro}.
\end{proof}

We remark that $K$ is universally locally acyclic relative to $f$ if and only if the pair $(f, K)$ is $\Psi$-good and $K$ is locally acyclic relative to $f$ \cite[Example~1.7(b)]{illusie2017around} \cite[Proposition~2.7(2)]{saito2017characteristic}.

For a composition of specializations $s \rightarrow t \rightarrow u$, there is a natural map
\begin{equation}\label{eqn: psi factorizability definition} R(\Psi_f)_{s}^{u} K \rightarrow R(\Psi_f)_{t}^{u} R(\Psi_f)_{s}^{t} K. \end{equation}
Following \cite[Definition~2.7]{salmon2023restricted}, we say that a pair $(f, K)$ is $\Psi$-factorizable if $(f, K)$ is $\Psi$-good and the above map is an isomorphism for any such composition of specializations.  The motivation for the terminology of $\Psi$-factorizability was to unify two sorts of ``factorization structures'' appearing in the study of Satake and affine Hecke categories: the commutativity of the convolution product on Satake sheaves with the centrality of the central sheaf functor into various parahoric affine Hecke categories under a single construction.  Let $G$ be a parahoric group scheme over a curve $C$ over $\F$, and let $\Gr_{G, I}$ be the Beilinson-Drinfeld Grassmannian with $I = \{ 1,\dots,n \}$ legs \cite[Definition~2.2]{salmon2023unipotent} over the base
\[ \mathfrak{q}_I \colon \Gr_{G,I} \rightarrow C^I. \]
We note that if $|I| = 1$ and $G$ has parahoric ramification at a point $x$, then the corresponding fiber of the Beilinson-Drinfeld Grassmannian over $x$ is a partial affine flag variety.  Let $\{ x_1, \dots, x_m \} = N \subset C$ be the locus of parahoric ramification and $C \setminus N$ the unramified locus.  There are Satake sheaves $\mathcal{S}_{V_1} \widetilde{\boxtimes} \dots \widetilde{\boxtimes} \mathcal{S}_{V_n}$ over $\Gr_{G,I} \times_{C^I} (C \setminus N)^I$, and we can $!$-extend them to sheaves over $\Gr_{G,I}$ via
\[ j \colon \Gr_{G,I} \times_{C^I} (C \setminus N)^I \rightarrow \Gr_{G,I}. \]
The following proposition is then a special case of \cite[Proposition~3.14]{salmon2023restricted}, announced there in the case of $\qlbar$ coefficients but whose proof extends to general choices of $\bk$, including characteristic $\ell$ coefficients.
\begin{proposition}\label{prop: psi factorizability grassmannian}
    The pairs $(\mathfrak{q}_I, j_!(\mathcal{S}_{V_1} \widetilde{\boxtimes} \dots \widetilde{\boxtimes} \mathcal{S}_{V_n}))$ are $\Psi$-factorizable.
\end{proposition}

On the other hand, Gabber and Abe have announced that $\Psi$-factorizability is already implied by $\Psi$-goodness.  The following result was announced as \cite[Theorem~4.5]{abe2022ramification} in the case of \'etale sheaves.
\begin{conjecture}\label{prop: factorizable}
    If $(f, K)$ is $\Psi$-good, then $(f, R(\Psi_f)_{u}^{t} K)$ is $\Psi$-good and $(f, K)$ is $\Psi$-factorizable.
\end{conjecture}
For the remainder of the paper, if we will assume that the above holds for pro-\'etale sheaves, then all instances of $\Psi$-factorizability can be reduced to those of $\Psi$-goodness.  For the pro-\'etale case, we do not expect new ideas will be needed beyond the \'etale case.

\subsection{Nearby cycles over an affine space}

We now specialize to the case that $S$ is an affine space.  For a finite set $P$, let $\A^P$ be the affine space of maps $P \rightarrow \A^1$, as a contravariant functor in finite sets $P$.  For a finite set $P$, let $P_*$ denote the pointed finite set $P \cup \{*\}$ with $* \not \in P$, and let $\FinS$ be the category of finite sets and $\FinS_*$ be the category of pointed finite sets.  Let $f \colon X \to \A^{P} = S$ be a separated scheme of finite type.  We may sometimes consider $f$ as defining a map over the base $\A^{P_*}$ under the identification $\A^P \cong \A^P \times \overline{0} \rightarrow \A^{P_*}$ where $\overline{0}$ is the obvious geometric point over the closed $\mathbb{F}$-point $0 \in \A^{\{*\}}$ (recalling that our base field is algebraically closed).

For a map of pointed finite sets, there is a generalized diagonal $\Delta_{\alpha} \colon \A^{Q_*} \to \A^{P_*}$.  By choosing a geometric point $\overline{\eta_Q}$ over the generic point $\eta^Q$ of $\A^Q$, we produce a geometric point of $\A^{P_*}$.  A map of pointed finite sets factors as the composition of a surjective and injective map, and we will mostly consider the case when $\alpha$ is surjective, remarking that in the case that $\alpha$ is injective, this geometric point would map to a geometric generic point over $\A^P$ as well.  More generally, if $t \rightarrow S_{(s)}$ is a specialization coming from a map of geometric points $t \rightarrow s$, the corresponding nearby cycles functor $R(\Psi)_{t}^{s}$ is simply pullback along the fiber $X_t \rightarrow X_s$.

In the case that $\alpha$ is surjective, we now specialize the theory of nearby cycles over general bases in relation to our setting where $S = \A^P$.  Let $\FinSs_*$ be the category of pointed finite sets and \emph{surjective} maps.  Let $\overline{\eta_P}$ be a geometric point over the Zariski-dense point $\eta^P \in \A^{P}$, and let $\overline{0}$ be a $\mathbb{F}$-valued geometric point over $0 \in \A^1 = \A^{\{*\}}$.  Let $\overline{\eta_P} \times \overline{0}$ be the geometric point in $\A^{P_*}$.  For every surjective map of pointed finite sets $P_* \to Q_*$, this defines a geometric point $\Delta_{\alpha}(\overline{\eta_Q} \times \overline{0})$ in $\A^{P_*}$ and we can compatibly choose specializations $\overline{\eta_P} \times \overline{0} \to \Delta_{\alpha}(\overline{\eta_Q} \times \overline{0})$ in $\A^{P_*}$ so they are compatible with composition in $\FinSs_*$ and therefore define a functor from the under category $P_* / \FinSs_*$ to the category of geometric points and specializations in the \'etale topos of $\A^{P_*}$.  Denote this functor by $F$, so $F(\alpha)$ is the specialization in the \'etale topos.

A specialization map $s \rightarrow t$ in the \'etale topos of $S$ determines a map on the fibers of $X$ over \'etale neighborhoods of $S$.  Letting $X_{(s)}$ be the fiber over the \'etale localization $S_{(s)}$ of $s$, then the specialization defines a map $X_{(s)} \to X_{(t)}$.  In the case of a geometric point $\eta$ over a generic point, $X_\eta = X_{(\eta)}$.  Let $j_{\alpha}$ be the map $X_{\overline{\eta_P} \times \overline{0}} \rightarrow X_{(\Delta_{\alpha}(\overline{\eta_Q} \times \overline{0}))}$, and define nearby cycles with respect to the specialization $F(\alpha)$ with the notation
\begin{equation}
    R(\Psi_f)_{F(\alpha)} K := R(\Psi_f)_{\overline{\eta_P} \times \overline{0}}^{\Delta_{\alpha}(\overline{\eta_Q} \times \overline{0})} = i_{\alpha}^* j_{\alpha,*} K.
\end{equation}
We may choose geometric points to be compatible with composition with respect to maps out of $P_*$.  In particular, there is an adjunction map
\begin{equation}
    R(\Psi_f)_{F(\beta \alpha)} K \rightarrow R(\Psi_f)_{\Delta_{\alpha}(\overline{\eta_Q} \times \overline{0})}^{\Delta_{\beta \alpha}(\overline{\eta_R} \times \overline{0})} R(\Psi_f)_{\overline{\eta_P} \times \overline{0}}^{\Delta_{\alpha}(\overline{\eta_Q} \times \overline{0})} K
\end{equation}
coming from \cite[Construction~2.6]{salmon2023restricted}.

We say that a specialization $F(\alpha)$ is one-dimensional if the corresponding surjective morphism of finite sets $\alpha \colon P_* \rightarrow Q_*$ satisfies $|P| - |Q| = 1$.  For a one-dimensional specialization, we must have $|\alpha^{-1}(q)| = 2$ for a unique $q \in Q$ or $|\alpha^{-1}(*)| = 2$.  If $F(\alpha)$ is one-dimensional, then $\Delta_{\alpha}(\overline{\eta_Q} \times \overline{0})$ is a geometric generic point over the generic point of a divisor $D$ of $\A^P$.  As a result, the \'etale localization $S_{(\Delta_{\alpha}(\overline{\eta_Q} \times \overline{0}))}$ is a henselian trait and the corresponding nearby cycles are classical one-dimensional nearby cycles (without shift).  It is a classical fact, due to Gabber, that for one-dimensional $F(\alpha)$, the functor $K \mapsto R(\Psi)_{F(\alpha)}(K)[-1]$ is t-exact for the perverse t-structure.

\begin{proposition}\label{prop: perverse}
    Let $f \colon X \to \A^P = S$ be the structure map, separated of finite type, and assume that the pair $(f, K)$ is $\Psi$-factorizable with $K$ a perverse sheaf.  For any surjection $\alpha \colon P_* \to Q_*$, the object $R(\Psi)_{F(\alpha)} K[|Q|-|P|]$ is a perverse sheaf.
\end{proposition}

\begin{proof}
    For a surjective map $\alpha \colon P_* \rightarrow Q_*$, we may write $F(\alpha)$ as a composition $F(\alpha_1) \circ \dots \circ F(\alpha_n)$ of one-dimensional specializations, where $n = |P| - |Q|$.  By inducting on $n$, the $\Psi$-factorizability of $(f, K)$ implies that the map
    \[ R(\Psi)_{F(\alpha)} K \rightarrow R(\Psi)_{F(\alpha_1)} \cdots R(\Psi)_{F(\alpha_n)} K \]
    is an isomorphism.  On the other hand, the right side is a composition of classical one-dimensional nearby cycles, so the result follows by the t-exactness of the one-dimensional nearby cycle functor.
\end{proof}

Every map $\alpha \colon P_* \rightarrow Q_*$ admits a unique factorization $\alpha^{\Delta} \circ \alpha^*$.  Here, $\alpha^* \colon P_* \rightarrow (P \setminus \alpha^{-1}(*))_*$ is defined by
\[
\alpha^*(i) = \begin{cases} i & i \notin \alpha^{-1}(*) \\ * & \text{ otherwise,} \end{cases}
\]
and $\alpha^{\Delta} \colon (P \setminus \alpha^{-1}(*))_* \rightarrow Q_*$ is defined by
\[
\alpha^{\Delta}(i) = \begin{cases} \alpha(i) & i \ne * \\ * & i = * \end{cases}
\]
In the case $\alpha = \alpha^*$, we wish to perform a deeper study of the Galois action.  The Galois group $\pi_1(\eta^P, \overline{\eta_P})$ acts on $R(\Psi)_{F(\alpha)} K$ by automorphisms of the geometric generic point.  There is a surjective map
\[
\pi_1(\eta^P, \overline{\eta_P}) \rightarrow \pi_1(\eta, \overline{\eta})^P
\]
which is not in general an isomorphism.  That is, for $p \in P$, the projection to $\pi_1(\eta, \overline{\eta})$ arises from mapping the automorphism of $\overline{\eta_P}$ to the automorphism of $\overline{\eta}$ induced by the projection along $\A^P \rightarrow \A^1$.

\begin{proposition}\label{prop: galois action}
    Let $f \colon X \to \A^P = S$ be the structure map, separated of finite type, and assume that the pair $(f, K)$ is $\Psi$-factorizable with $K$ a perverse sheaf.  For any surjection $\alpha \colon P_* \to Q_*$ that is a bijection outside $\alpha^{-1}(*) \cap P = \{ p_1, \dots, p_n \}$, there is an isomorphism
    \[ R(\Psi)_{F(\alpha)} K \rightarrow R(\Psi)_{p_1} \cdots R(\Psi)_{p_n} K, \]
    where $R(\Psi)_{p_i}$ is a one-dimensional nearby cycle with respect to the projection $\A^P \rightarrow \A^1$ projecting to the $p_i$ copy of $\A^1$.
    
    Moreover, this map on nearby cycles is equivariant with respect to the Galois action of $\pi_1(\eta^P, \overline{\eta_P})$ on the left mapping to the action of $\pi_1(\eta, \overline{\eta})^{\{ p_1, \dots, p_n\}}$ on the right.
\end{proposition}

\begin{proof}
    Write $\alpha = \alpha_{p_1} \circ \dots \circ \alpha_{p_n}$ where $\alpha_{p_i} \colon P_* \setminus \{ p_{i+1}, \dots, p_n \} \rightarrow P_* \setminus \{ p_{i}, \dots, p_n \}$ defined by
    \[ \alpha_{p_i}(p) = \begin{cases} * & p = p_i \\ p & \text{ otherwise.} \end{cases} \]
    By inducting on $\Psi$-factorizability, we may conclude that the map
    \[ R(\Psi)_{F(\alpha)} K \rightarrow R(\Psi)_{F(\alpha_{p_1})} \cdots R(\Psi)_{F(\alpha_{p_n})} K \]
    is an isomorphism.  We next argue that each nearby cycle $R(\Psi)_{F(\alpha_{p_i})}$ is a one-dimensional nearby cycle $R(\Psi)_{p_i}$.  Let $Q_i = P \setminus \{ p_{i}, \dots, p_n \}$ and let $\beta_i = \alpha_{i} \circ \dots \circ \alpha_n \colon P_* \rightarrow (Q_i)_*$ with $\beta_n = \alpha_n$ and $\beta_1$ equal to the map $\alpha$.  The specialization $F(\alpha_{p_i})$ can be viewed as a specialization
    \[ \Delta_{\beta_{i+1}}(\overline{\eta_{Q_{i+1}}} \times \overline{0}) \rightarrow \Delta_{\beta_i}(\overline{\eta_{Q_i}} \times \overline{0}) \]
    in the \'etale topos of $S$.  We note that the specialization map
    \[ \overline{\eta_{Q_{i+1}}} \times \overline{0} \rightarrow \A^P_{(\Delta_{\beta_{i}}(\overline{\eta_{Q_{i}}} \times \overline{0}))} \]
    factors as the composition of
    \[ \overline{\eta_{Q_i}} \times \overline{0} \rightarrow \A^{Q_{i+1}}_{(\Delta_{\alpha_i}(\overline{\eta_{Q_{i}}} \times \overline{0}))} \rightarrow \A^P_{(\Delta_{\beta_{i}}(\overline{\eta_{Q_{i}}} \times \overline{0}))} \]
    where the second map is a closed embedding.  Therefore, the sheaf $R(\Psi)_{F(\alpha_i)}$ can be viewed as nearby cycles for the first specialization in the composition above, viewed as a specialization in the \'etale topos of $\A^{Q_{i+1}}$.  On the other hand, $\A^{Q_{i+1}}_{(\Delta_{\alpha_i}(\overline{\eta_{Q_{i}}} \times \overline{0}))}$ is a henselian discrete valuation ring, and projection to the $p_i$ coordinate gives a map
    \[ \A^{Q_{i+1}}_{(\Delta_{\alpha_i}(\overline{\eta_{Q_{i}}} \times \overline{0}))} \rightarrow \A^1_{(\overline{0})} \]
    so this specialization can be viewed as the specialization for nearby cycles on the localization of $\overline{0}$ in $\A^1$.  The identification in the first part now follows by noting that classical nearby cycles are independent of the choice of geometric points over the generic and special points of a henselian trait.  The action of local inertia on classical nearby cycles comes from the action of the Galois group of the geometric generic point via projection.
\end{proof}

As a consequence of the Galois equivariance in the previous proposition, the action of $\pi_1(\eta, \overline{\eta})^{\{ p_1, \dots, p_n\}}$ on $R(\Psi)_{F(\alpha)} K$ is independent of the choice of decomposition $\alpha = \alpha_{p_1} \circ \dots \circ \alpha_{p_n}$.

\section{Relation between unipotent nearby cycles and nearby cycles via the vanishing topos}

Since we primarily work with the perverse t-structure on the derived category of constructible sheaves, we will often drop the derived functor notation and write instead $j_* = Rj_*$, $(\Psi_f)_{F(\alpha)} = R(\Psi_f)_{F(\alpha)}$, etc.  As a rule, all functors will be derived in the derived category of constructible sheaves and perverse cohomologies will be noted explicitly when taken.

\subsection{The definition of Achar and Riche}

For a map $\alpha \colon P_* \to Q_*$ in $\FinS_*$, and $K$ a perverse sheaf on $X$, Achar-Riche have conditionally defined an object $\Upsilon^{\alpha}_f(K)$ \cite[Definition~2.3]{achar2023higher}.  We recall their definition of $\Upsilon^{\alpha}_f(K)$.  This diagonal $\Delta_{\alpha}$ induces a pullback map
\[\bi_{X,\alpha} \colon X \times_{\A^{P} \times \overline{0}} \A^Q \times \overline{0} \rightarrow X.\]
We often suppress the $X$ where the structure map $f \colon X \rightarrow \A^P$ is understood.  Let $j$ be the open embedding of $X \times_{\A^P} (\G_m)^P \rightarrow X$.  More generally, we write $\bj_{X}$ to specify the dependence of $j$ on the scheme $X$ over $\A^P$.

On these affine spaces we can define extensions of certain unipotent local systems, and we recall the setup of \cite[Section~2.5]{achar2023higher}.  First, on $\G_m = \A^1 \setminus \{ 0 \}$, there are unipotent local systems $\mathscr{L}_a$, indecomposable of rank $a$.  These are unique up to isomorphism, $\mathscr{L}_1$ is the constant local system, and for $a \le b$, there is an exact sequence
\begin{equation*}
    0 \rightarrow \mathscr{L}_a \rightarrow \mathscr{L}_b \rightarrow \mathscr{L}_{b-a} \rightarrow 0.
\end{equation*}
These local systems $\mathscr{L}_a$ come from representations of the tame inertia group $\Zhat' = \prod_{\ell \ne p} \Z_{\ell}$ that send a topological generator to an indecomposable $a \times a$ Jordan block.  We call these representations $L_a$.

For a map $\ba \colon P = \{ p_1, \dots, p_n \} \rightarrow \Z_{\ge 1}$, we form a local system $\mathscr{L}_{\ba}$ on $\G_m^P$ as an external product
\[ \mathscr{L}_{\ba} = \mathscr{L}_{\ba(p_1)} \boxtimes \dots \boxtimes \mathscr{L}_{\ba(p_n)}. \]
Under the map $f$, we may pull this back to a local system on $X$.  Similarly, we form representations
\[ L_{\ba} = L_{\ba(p_1)} \boxtimes \dots \boxtimes L_{\ba(p_n)}. \]
The set of such maps $\ba$ form a poset, and $\mathscr{L}_{\ba}$ and $L_{\ba}$ are functors from this poset into local systems on $\G_m^P$ and representations of $\Z^P$, respectively.

\begin{definition}
    Let $K$ be a perverse sheaf on $X \times_{\A^P} \G_m^P, \bk$.  Let $\Upsilon^{\alpha}_f(K) \in D^b(X_s, \bk)$ be the pro-\'etale complex of sheaves
    \begin{equation}
        \varinjlim_{\ba \colon \alpha^{-1}(*) \cap P \rightarrow \Z_{\ge 1}} \bi_{\alpha}^* j_{!*} ( K \otimes f^* \mathscr{L}_{\ba} )[|Q|-|P|]
    \end{equation}
    where the limit is taken over the poset of maps $\alpha^{-1}(*) \cap P \rightarrow \Z_{\ge 1}$.  For $K \in D^b_c(X, \bk)$, we may also write $\Upsilon^{\alpha}_f(K)$ where we implicitly restrict to $X \times_{\A^P} \G_m^P$.
\end{definition}

Later, we will relate the functorial description above to Achar and Riche's original definition.  If $K$ is perverse, Achar and Riche define $\Upsilon^{\alpha}_f(K)$ as
\begin{equation}
    \varinjlim_{\ba \colon \alpha^{-1}(*) \cap P \to \Z_{\ge 1}} {}^{p} \scH^{|Q|-|P|} \bi_{\alpha}^* j_* ( K \otimes f^* \mathscr{L}_{\ba} ),
\end{equation}
conditional on two conditions, the first of which stating that the limit stabilizes for $\ba$ sufficiently large and the second stating that the maps outside perverse degree $|Q|-|P|$ vanish for $\ba$ sufficiently large.  We will see later in Corollary~\ref{cor: unipotent general bases ula} that if $K$ is perverse and $(f, K)$ is $\Psi$-good and universally locally acyclic over $\G_m^P$, then these definition agree in the sense that $\Upsilon^{\alpha}_f(K)$ is a constructible sheaf and its shift by $|Q|-|P|$ is perverse.

\subsection{Tame and unipotent nearby cycles via the vanishing topos}

We keep the notation of the previous section.  In particular, $\alpha$ will always be a surjective map of pointed finite sets that factors as $\alpha^{\Delta} \circ \alpha^*$ where $\alpha^*$ is a bijection outside the preimage of $*$.  We will often handle maps $\alpha$ that induce such bijectins separately and denote this condition as $\alpha = \alpha^*$.

For the localization of $\A^1$ at $0$, the corresponding Galois group $\pi_1(\eta, \overline{\eta})$ is the local inertia, and it has a distinguished quotient $\Z'$, where $\Z' = \prod_{\ell \ne p} \Z_\ell$ is the Galois group of the inverse limit of Kummer coverings $\widetilde{\G_m} \rightarrow \G_m$.  We will choose a topological generator of $\Z'$ along each coordinate $p \in P$, which we denote $T_p$.  Let $W$ denote the wild inertia, the kernel of the map $\pi_1(\eta, \overline{\eta}) \rightarrow \Z'$.  The previous proposition shows that if $\alpha = \alpha^*$ and if $(f, K)$ is $\Psi$-factorizable with $K$ a perverse sheaf, we may consider $R(\Psi)_{F(\alpha)} K$ as a perverse sheaf equipped with an action of $\pi_1(\eta, \overline{\eta})^{\alpha^{-1}(*) \cap P}$.  Since $W^{\alpha^{-1}(*) \cap P}$ is a pro-p group, it is meaningful to take $W^{\alpha^{-1}(*) \cap P}$-invariants as a perverse sheaf.

\begin{definition}
Let $\alpha = \alpha^*$ be a surjective map $P_* \rightarrow Q_*$.  If $(f, K)$ is $\Psi$-factorizable and $K$ is perverse on $X$, we define tame nearby cycles as the perverse sheaf
\begin{equation}
(\Psi^t_f)_{F(\alpha)} K = ((\Psi_f)_{F(\alpha)} K)^{W^{\alpha^{-1}(*) \cap P}}.
\end{equation}
The action of $\pi_1(\eta, \overline{\eta})^{\alpha^{-1}(*) \cap P}$ further factors through the quotient $(\Z')^{\alpha^{-1}(*) \cap P}$ on this subobject.  We define unipotent nearby cycles as the subobject of $(\Psi^t_f)_{F(\alpha)} K$ for which $T_p$ act unipotently for all $p \in \alpha^{-1}(*) \cap P$.  To be more explicit, the space of endomorphisms of a perverse sheaf is finite-dimensional.  Therefore, we may consider $(\Psi_f^t)_{F(\alpha)} K$ as a module over $\bk[T_{p_1}, \dots, T_{p_n}]$.  Analogous to \cite[Lemma~1.1]{reich2010notes}, \cite[Proposition~1.1]{morel2018beilinson} which deals with the case of a one-dimensional base, this submodule decomposes into a part supported over $(1, \dots, 1) \in \Spec(\bk[T_1, \dots, T_n])$, which we denote $(\Psi_f^u)_{F(\alpha)} K$, and a part supported away from this point, which we denote $(\Psi_f^{\ne 1})_{F(\alpha)} K$, yielding a canonical direct sum decomposition into generalized eigenspaces
\begin{equation}
    (\Psi_f^t)_{F(\alpha)} K = (\Psi_f^u)_{F(\alpha)} K \oplus (\Psi_f^{\ne 1})_{F(\alpha)} K.
\end{equation}
\end{definition}

As a consequence of Proposition~\ref{prop: galois action}, for $\alpha^{-1}(*) \cap P = \{ p_1, \dots, p_n \}$, if $(f, K)$ is $\Psi$-factorizable, there are canonical isomorphisms
\[ (\Psi^t_f)_{F(\alpha)} K \rightarrow (\Psi^t)_{p_1} \cdots (\Psi^t)_{p_n} K, \]
where tame nearby cycles over a curve $(\Psi^t)_{p_i} L$ are understood as the wild inertia invariants $((\Psi)_{p_i} L)^W$.  Similarly, under the same assumptions there are canonical isomorphisms
\[ (\Psi^u_f)_{F(\alpha)} K \rightarrow (\Psi^u)_{p_1} \cdots (\Psi^u)_{p_n} K. \]
In particular, for such pairs $(f, K)$, iterated tame and unipotent nearby cycles are independent of the choice of permutation, so that, for example,
\[ (\Psi^u)_{p_1} \cdots (\Psi^u)_{p_n} K \cong (\Psi^u)_{p_{\sigma(1)}} \cdots (\Psi^u)_{p_{\sigma(n)}} \]
for any permutation $\sigma$ of $\{1,\dots,n\}$.

In the study of unipotent nearby cycles over a one-dimensional base, there is a fundamental exact triangle
\[
\begin{tikzcd}
i^* j_* \arrow{r}& R\Psi^u \arrow{r}{(1 - T)}& R\Psi^u \arrow{r}{+1}& {}
\end{tikzcd}
\]
where $i$ and $j$ are the inclusion of the fiber over $0 \rightarrow \A^1$ and the complement, respectively.  We will want a version of this exact triangle over the base $\A^P$.  One important difference is that the complementary open embedding to $\bi_{\alpha}$ is not affine.  As a result, we may not apply the results of \cite[Section~4.1]{beilinson1982faisceaux} directly.  The following lemmas substitute for the necessary results.
\begin{lemma}\label{lem: ij push pull}
    Let $f \colon X \rightarrow \A^P$ separated of finite type over an algebraically closed base field with $K$ a perverse sheaf on $X$ such that the pair $(f, K)$ is $\Psi$-factorizable.  Let $\alpha = \alpha^* \colon P_* \rightarrow Q_*$ be a surjective map.  Then
    \[ {}^p \scH^{|Q|-|P|} \bi_{\alpha}^* \bj_{*} K \rightarrow (\Psi)_{F(\alpha)} K \]
    factors through $(\Psi^t)_{F(\alpha)} K$ and defines an isomorphism onto the $(\Z')^{\alpha^{-1}(*) \cap P}$-invariant part.
\end{lemma}

\begin{proof}
    By \cite[Lemma~2.6]{achar2023higher}, $\bi^*_{\alpha} \bj_{*} K$ lives in perverse degrees $\ge |Q|-|P|$.

    On the other hand, we may relate this complex to derived global sections in a way parallel to \cite[Section~1.2]{saito2021characteristic}.  Let $\xi$ be the generic point of the localization of $\A^P$ at $\Delta_{\alpha}(\overline{\eta_Q} \times \overline{0})$.  Let $I$ be the Galois group $\Gal(\overline{\xi} / \xi)$, noting that we may choose $\overline{\xi}$ to be identified with the geometric point $\overline{\eta_Q}$.  We have the following commutative diagram:
    \begin{equation}
        \begin{tikzcd}
            \overline{\eta_Q} \arrow{r}{\pi}& \xi \arrow{r}{j_{\xi}} \arrow{d}{h}& \A^P_{(\Delta_{\alpha}(\overline{\eta_Q} \times \overline{0}))} \arrow{d}{h} & \Delta_{\alpha}(\overline{\eta_Q} \times \overline{0}) \arrow{l}{i_{\alpha}} \arrow{d}{h} \\
            & \G_m^P \arrow{r}{j}& \A^P & \Delta_{\alpha}(\A^Q) \arrow{l}{\bi_{\alpha}}
        \end{tikzcd}
    \end{equation}
    where $\pi$ has Galois group $I$ and all other maps are natural ones, with the notation $h$ used for the vertical maps, by abuse of notation.  When we pull back the above diagram along $f \colon X \rightarrow \A^P$, nearby cycles can be considered as a functor $(\Psi)_{F(\alpha)} \pi^* K = i_{\alpha}^* (j_{\xi})_*\pi_* \pi^* K$, so taking the derived $I$-invariants give the identification
    \[ R\Gamma(I, (\Psi)_{F(\alpha)} \pi^* K) \cong i_{\alpha}^* (j_{\xi})_* K. \]
    If $(f, K)$ is $\Psi$-factorizable, then the action of $I$ factors through the power of inertia on each copy of one-dimensional nearby cycles and gives
    \[ R\Gamma(I, \Psi_{F(\alpha)} \pi^* h^* K) \cong i_{\alpha}^* (j_{\xi})_* h^* K \cong h^* \bi_{\alpha} j_* K, \]
    where the last isomorphism follows by writing the localization as the inverse limit of \'etale maps and applying the smooth base change theorem \cite[Lemma~59.89.3]{stacks-project}.

    Since $(\Psi)_{F(\alpha)} K[|Q|-|P|]$ is perverse, in perverse degree $|Q|-|P|$, the $I$-action factors through the $\pi^{-1}(*)$-power of the local inertia of $\A^1_{(0)}$.  Therefore, the $I$-invariants of $(\Psi)_{F(\alpha)} K$ are the same as $(\Z')^{\pi^{-1}(*)}$-invariants of $(\Psi^t)_{F(\alpha)} K$.
\end{proof}

For the next lemma, we introduce some notation for the relevant maps.  For a map $\alpha \colon P_* \rightarrow Q_*$, we define the $\alpha$-generic part of $\A^P$, $\A^P_{\alpha}$ as the subset of $\A^P$ such that $x_p \ne 0$ for all $\alpha(p) \ne *$.  Define $\bj_{\alpha} \colon \A^P_{\alpha} \rightarrow \A^P$ as the open embedding and $\bj_{X,\alpha}$ to be the pullback of this map to any $X$ over $\A^P$.  For any $X$ over $\A^P$, we also define $X_{\alpha} = X \times_{\A^P} \A^P_{\alpha}$, equipped with structure maps $X_{\alpha} \rightarrow \A^P_{\alpha} \rightarrow \A^P$.  For surjective maps $\alpha = \alpha^* \colon P_* \rightarrow Q_*$ and $\beta = \beta^* \colon Q_* \rightarrow R_*$, a composition $\beta \circ \alpha$, the map $\bi_{\A^P_{\beta \alpha}, \beta \circ \alpha}$ factors as
\begin{equation}
    \begin{tikzcd}
        \G_m^R \arrow{r}{\bi_{\A^Q_{\beta}, \beta}} & \A^Q_{\beta} \arrow{r}{\bi_{\A^P_{\beta \alpha}, \alpha}} & \A^P_{\beta \alpha},
    \end{tikzcd}
\end{equation}
so there is a diagram
\begin{equation}\label{eqn: composite diagram}
    \begin{tikzcd}
        & & \G_m^R \arrow{d}{\bi_{\A^Q_{\beta}, \beta}} \\
        & \G_m^Q \arrow{r}{\bj_{\A^Q_{\beta}}} \arrow{d}{\bi_{\A^P_{\alpha}, \alpha}} & \A^Q_{\beta} \arrow{d}{\bi_{\A^P_{\beta \alpha}, \alpha}} \\
        \G_m^P \arrow{r}{\bj_{\A^P_{\alpha}}} & \A^P_{\alpha} \arrow{r}{\bj_{\A^P_{\beta \alpha}, \alpha}} & \A^P_{\beta \alpha}
    \end{tikzcd}
\end{equation}

\begin{lemma}\label{lem: ij middle extension}
    Under the same assumptions as Lemma~\ref{lem: ij push pull},
    \[ \bi_{\alpha}^* \bj_{!*} K \cong ((\Psi^u_f)_{F(\alpha)} K)^{(\Z')^{\alpha^{-1}(*) \cap P}}. \]
\end{lemma}

\begin{proof}
    We induct on $|\alpha^{-1}(*)|$ and handle the case of a composition $\beta \alpha$ with both $\beta = \beta^*$ and $\alpha = \alpha^*$.  We note that the one-dimensional case $|\alpha^{-1}(*)| = 2$ was proved by Beilinson \cite{beilinson1987glue}.  Diagram~\ref{eqn: composite diagram} produces a diagram
    \begin{equation}
    \begin{tikzcd}
        \bi_{\A^Q_{\beta}, \beta}^* \bi_{\A^P_{\beta \alpha}, \alpha}^* (\bj_{\A^P_{\beta \alpha}, \alpha})_! (\bj_{\A^P_{\alpha}})_{!*} K \arrow{r} \arrow{d} & \bi_{\A^Q_{\beta}, \beta}^* (\bj_{\A^Q_{\beta}})_! \bi_{\A^P_{\alpha}, \alpha}^* (\bj_{\A^P_{\alpha}})_{!*} K \arrow{d} \\
        \bi_{\A^Q_{\beta}, \beta}^* \bi_{\A^P_{\beta \alpha}, \alpha}^* (\bj_{\A^P_{\beta \alpha}, \alpha})_* (\bj_{\A^P_{\alpha}})_{!*} K \arrow{r}& \bi_{\A^Q_{\beta}, \beta}^* (\bj_{\A^Q_{\beta}})_* \bi_{\A^P_{\alpha}, \alpha}^* (\bj_{\A^P_{\alpha}})_{!*} K
    \end{tikzcd}
    \end{equation}
    where the top horizontal arrow is an isomorphism.  The inductive hypothesis implies that
    \[ \bi_{\A^P_{\alpha}, \alpha}^* (\bj_{\A^P_{\alpha}})_{!*} K[|Q|-|P|] \] 
    is perverse, so we may make sense of
    \[ \bi_{\A^Q_{\beta}, \beta}^* (\bj_{\A^Q_{\beta}})_{!*} \bi_{\A^P_{\alpha}, \alpha}^* (\bj_{\A^P_{\alpha}})_{!*} K[|Q|-|P|] \]
    as the $\bi_{\A^Q_{\beta},\beta}^*$ pullback of the image of perverse sheaves defining the intermediate extension in the right arrow.  Since the diagram commutes, we conclude that $\bi_{\alpha}^* \bj_{!*} K \cong \bi_{\A^Q_{\beta}, \beta}^* (\bj_{\A^Q_{\beta}})_{!*} \bi_{\A^P_{\alpha}, \alpha}^* (\bj_{\A^P_{\alpha}})_{!*} K$.  On the other hand, this object is
    \[ ((\Psi^u_f)_{F(\beta)} (\Psi)_{F(\alpha)} K)^{(\Z')^{\alpha^{-1}(*) \cap P} \times (\Z')^{\beta^{-1}(*) \cap Q}} \cong ((\Psi^u_f)_{F(\beta \alpha)} K)^{(\Z')^{(\beta \alpha)^{-1}(*) \cap P}}. \]
\end{proof}

\subsection{Relation of the two definitions}

\begin{lemma}\label{lem:L relation}
    The pair $(1, \mathscr{L}_{\ba})$ is $\Psi$-factorizable, where $1$ is the identity map.  If $\ba(p) = 1$ for all $p \not \in \alpha^{-1}(*) \cap P$, there are isomorphisms,
    \begin{equation}
        (\Psi_1)_{F(\alpha)}(\mathscr{L}_{\ba}) = (\Psi^{u}_1)_{F(\alpha)}(\mathscr{L}_{\ba}) = L_{\ba}
    \end{equation}
    as representations of $(\Zhat')^{\alpha^{-1}(*) \cap P}$ generated by $T_{p}$ for $p \in \alpha^{-1}(*) \cap P$.
\end{lemma}

\begin{proof}
    Since $\mathscr{L}_{\ba}$ is an external product, this reduces to the one-dimensional case by the K\"unneth formula \cite[Theorem~2.3]{illusie2017around} \cite[Proposition~2.10]{salmon2023restricted}.
\end{proof}

We now arrive at the connection between $(\Psi_f^{u})_{F(\alpha)} K$ and $\Upsilon^{\alpha}_f(K)$.  This will generalize \cite[Proposition~3.1]{morel2018beilinson}.
\begin{proposition}
    Let $(f,K)$ be $\Psi$-factorizable with $K$ a perverse sheaf.  Let $\alpha = \alpha^* \colon P_* \rightarrow Q_*$ be surjective with $\alpha^{-1}(*) \cap P = \{ p_1, \dots, p_n \}$.  Let $\ba \colon \alpha^{-1}(*) \cap P \rightarrow \Z_{\ge 1}$.  Then
    \begin{equation}\label{eqn: nilpotent kernel left}
        \bigcap_{\ell \in \alpha^{-1}(*) \cap P} \ker ((1-T_{\ell})^{\ba(\ell)}, (\Psi_f^u)_{F(\alpha)} K)
    \end{equation}
    is isomorphic to
    \begin{equation}\label{eqn: nilpotent kernel right}
        \bigcap_{\ell \in \alpha^{-1}(*) \cap P} \ker (1-T_{\ell}, (\Psi_f^u)_{F(\alpha)} (K \otimes f^* \mathscr{L}_{\ba})).
    \end{equation}
\end{proposition}
\begin{proof}
    To make sense of the right hand side as a perverse sheaf, we note that the fact that the pair $(f, K \otimes f^*\mathscr{L}_{\ba})$ is $\Psi$-factorizable follows by the projection formula \cite[Example~4.26~(5)]{lu2019duality} \cite[Proposition~A.6]{illusie2017around}.

    For a given $p \in P$, we define
    \[\ba^{(p)}(j) = \begin{cases} \ba(j) & j \ne p, \\ 1 & j = p. \end{cases} \]
    Identifying the space of maps $\ba \colon P \to \Z$ with $\Z^P$ with addition and letting $e_p$ be the unit vector in the $p$ coordinate, there is an exact sequence
    \[ 0 \rightarrow \mathscr{L}_{\ba} \rightarrow \mathscr{L}_{\ba+e_p} \rightarrow \mathscr{L}_{\ba^{(p)}} \rightarrow 0. \]
    
    Using this exact sequence, we first establish an isomorphism $(\Psi^u_f)_{F(\alpha)}(K \otimes f^* \mathscr{L}_{\ba}) \cong (\Psi^u_f)_{F(\alpha)}(K) \otimes L_{\ba}$.  In the base case $\ba(p) = 1$ for all $p$ this is a tautology, while the inductive step is built from the commutative diagram
    \begin{equation}
        \begin{tikzcd}
            (\Psi_{f}^{u})_{F(\alpha)}(K) \otimes L_{\ba} \arrow{r} \arrow{d} & (\Psi^u_f)_{F(\alpha)}(K \otimes f^* \mathscr{L}_{\ba}) \arrow{d} \\
            (\Psi_{f}^{u})_{F(\alpha)}(K) \otimes L_{\ba+e_{p}} \arrow{r} \arrow{d} & (\Psi^u_f)_{F(\alpha)}(K \otimes f^* \mathscr{L}_{\ba + e_{p}}) \arrow{d} \\
            (\Psi_{f}^{u})_{F(\alpha)}(K) \otimes L_{\ba^{(p)}} \arrow{r} & (\Psi^u_f)_{F(\alpha)}(K \otimes f^* \mathscr{L}_{\ba^{(p)}}),
        \end{tikzcd}
    \end{equation}
    noting that the horizontal arrows are compositions of canonical maps coming from pullback and tensor product, as well as Lemma~\ref{lem:L relation}.  Since the vertical columns are distinguished triangles and the top and bottom rows are isomorphisms, we conclude the middle row is an isomorphism as well.

    We are now ready to prove the proposition.  There is a vector $v$ of $L_\ba$ such that, as a vector space, $L_{\ba}$ is generated by $\prod (1-T_{p_k})^{j_k} v := v_{j_1,\dots,j_n}$ for $0 \le j_i < \ba(p_i)$.  This gives a distinguished basis $v_{j_1,\dots,j_n}$ of $L_{\ba}$ as a vector space.
    
    Define a map $\gamma \colon (\Psi^u_f)_{F(\alpha)}(K) \rightarrow (\Psi^u_f)_{F(\alpha)}(K) \otimes L_{\ba}$ by
    \begin{equation}
        \gamma(x) = \sum_{j_1 = 0}^{\ba(p_1)-1} \cdots \sum_{j_n=0}^{\ba(p_n)-1} \left(\prod_{i=1}^n (-(1-T_{p_i}))^{\ba(p_i) - 1 - j_{i}} x\right) v_{j_1,\dots,j_m}.
    \end{equation}
    We note that $\gamma$ is injective because the coefficient of $v_{\ba(p_1)-1, \dots, \ba(p_n)-1}$ is $x$.  A straightforward computation shows that $(1-T_{\ell}) \gamma(x)$ is equal to
    \begin{equation}
        \sum_{j_1,\dots,j_{\ell-1},j_{\ell+1}, \dots, j_n} \left(\prod_{p_i \ne \ell} (-(1-T_{p_i}))^{\ba(p_i) - 1 - j_i} \cdot (-(1-T_{\ell}))^{\ba(\ell)} x\right) v_{j_1,\dots,j_{\ell-1},0,j_{\ell+1},\dots,j_n}.
    \end{equation}
    By considering when this expression is equal to $0$, we may show that $x$ is in the kernel of $(1-T_{\ell})^{\ba(\ell)}$ if and only if $\gamma(x)$ is in the kernel of $(1-T_{\ell})$.
    
    To show the isomorphism, it suffices to show that $\gamma$ surjects onto the intersection of the kernels of $(1-T_{\ell})$ for $\ell \in \alpha^{-1}(*) \cap P$.  Writing $\ell = p_i$ and $y \in (\Psi^u_f)_{F(\alpha)}(K) \otimes L_{\ba}$ as $\sum a_{j_1,\dots,j_n} v_{j_1,\dots,j_n}$, the condition that $y$ is in the kernel of $1-T_{\ell}$ is equivalent to relations $a_{j_1, \dots, j_{i-1}, j_i+1, j_{i+1}, \dots, j_n} = -(1-T_{\ell}) a_{j_1, \dots, j_n}$ and $a_{j_1, \dots, j_{i-1}, 0, j_{i+1}, \dots, j_n} = 0$.  If $y$ is in the intersection of all the kernels of $1-T_{\ell}$ for $\ell \in \alpha^{-1}(*) \cap P$, then $y$ is in the image of $\gamma(x)$ as $y$ can be determined uniquely by the single coefficient $a_{\ba(p_1)-1, \dots, \ba(p_n)-1}$ and $y = \gamma(a_{\ba(p_1)-1, \dots, \ba(p_n)-1})$.
    
    So $\gamma$ restricts to an isomorphism between the kernel of $(1-T_{\ell})^{\ba(\ell)}$ in $(\Psi^u_f)_{F(\alpha)} K$ and the kernel of $(1-T_{\ell})$ in $(\Psi^u_f)_{F(\alpha)} (K \otimes f^* \scL_{\ba})$.
\end{proof}

\begin{corollary}\label{cor: unipotent general bases}
    Let $(f,K)$ be $\Psi$-factorizable and let $\alpha = \alpha^* \colon P_* \rightarrow Q_*$ be surjective.  Then $\varinjlim {}^p \scH^{|Q|-|P|} \bi_{\alpha}^* j_*(K \otimes \mathscr{L}_{\ba})$ exists as a perverse sheaf, stabilizing for $\ba$ sufficiently large (one of the two necessary conditions in \cite[Definition~2.3]{achar2023higher}).  Moreover, this quantity coincides with $\Upsilon^{\alpha}_f(K) \cong (\Psi_f^{u})_{F(\alpha)} K[|Q|-|P|]$.
\end{corollary}
\begin{proof}
    We first prove the isomorphism.  The left hand side in Equation~\eqref{eqn: nilpotent kernel left} must stabilize for $\ba$ sufficiently large and equal $(\Psi_f^{u})_{F(\alpha)} K[|Q|-|P|]$ for $\ba$ sufficiently large.  After taking this intersection, the right hand side in Equation~\eqref{eqn: nilpotent kernel right} is the same as taking invariants under the action of $(\Z')^{\alpha^{-1}(*) \cap P}$, so becomes ${}^p \scH^{|Q|-|P|} \bi_{\alpha}^* j_* (K \otimes \mathscr{L}_{\ba})$ for such $\ba$ by Lemma~\ref{lem: ij push pull}, isomorphic also to $\bi_{\alpha}^* j_{!*} (K \otimes \mathscr{L}_{\ba})[|Q|-|P|]$ by Lemma~\ref{lem: ij middle extension}.  Therefore, for this choice of $\ba$,
    \[ \Upsilon^{\alpha}_f(K) \cong \bi_{\alpha}^* j_* (K \otimes \mathscr{L}_{\ba}) \cong (\Psi_f^{u})_{F(\alpha)} K. \]
    We also note that for $\ba \colon \alpha^{-1}(*) \cap P \to \Z_{\ge 1}$ sufficiently large, there is an isomorphism
    \[ \Upsilon^{\alpha}_f(K) \cong {}^p \scH^{|Q|-|P|} \bi_{\alpha}^* j_* ( K \otimes f^* \mathscr{L}_{\ba} ) \]
    so the relevant limit stabilizes at a finite level.
\end{proof}

Recall now that for general $\alpha \colon P_* \rightarrow Q_*$ surjective, $\bi_{\alpha}^*$ factors as $\bi_{\alpha^{\Delta}}^* \circ \bi_{\alpha^*}^*$.  A priori, the functor $\bi_{\alpha^{\Delta}}^*[|Q|-|P \setminus \alpha^{-1}(*)|]$ is not t-exact for the perverse t-structure, but only right exact.  However, under the assumption that $(f, K)$ is universally locally acyclic over $\G_m^P$, we can recover the perversity of $\Upsilon^{\alpha}_f(K)$.

\begin{corollary}\label{cor: unipotent general bases ula}
    Let $(f,K)$ be $\Psi$-factorizable and universally locally acyclic with respect to the restriction of $f$ over $\G_m^P \rightarrow \A^P$.  Let $\alpha \colon P_* \rightarrow Q_*$ be surjective.  Then $\varinjlim {}^p \scH^{|Q|-|P|} \bi_{\alpha}^* j_*(K \otimes \mathscr{L}_{\ba})$ coincides with $\Upsilon^{\alpha}_f(K) \cong (\Psi_f^{u})_{F(\alpha)} K[|Q|-|P|]$.
\end{corollary}

\begin{proof}
    This follows from Corollary~\ref{cor: unipotent general bases} if we show that the map
    \begin{equation}\label{eqn: local acyclic arrow}
    \begin{aligned}
        \bi_{\alpha^{\Delta}}^* \Upsilon^{\alpha}_f(K)[|Q|-|P \setminus \alpha^{-1}(*)|] &\rightarrow (\Psi_f)_{F(\alpha^{\Delta})} \Upsilon^{\alpha}_f(K)[|Q|-|P\setminus \alpha^{-1}(*)|] \\
        &\cong (\Psi_f)_{F(\alpha^{\Delta})} (\Psi_f^u)_{F(\alpha^*)} K[|Q|-|P|]
    \end{aligned}
    \end{equation}
    is an isomorphism.  To see this, the right side of this arrow is perverse since it is isomorphic to an iterated one-dimensional specialization along successive Henselian discrete valuation rings and using the t-exactness of one-dimensional nearby cycles.  By applying Corollary~\ref{cor: unipotent general bases}, this implies that
    \[ \bi_{\alpha^{\Delta}}^* {}^p \scH^{|Q|-|P \setminus \alpha^{-1}(*)|} \bi_{\alpha^*}^* j_* (K \otimes f^* \mathscr{L}_{\ba}) \]
    is concentrated in perverse degree $|\alpha^{-1}(*)|-1$ for sufficiently large $\mathscr{L}_{\ba}$ and is isomorphic to
    \[ {}^p \scH^{|Q|-|P|} \bi_{\alpha}^* j_* (K \otimes f^* \mathscr{L}_{\ba}). \]

    We now show the desired property in Equation~\ref{eqn: local acyclic arrow}.  We write $Q \sqcup \alpha^{-1}(*)$ as a disjoint union, picking new representatives if there is a nonempty set-theoretic intersection.  We may write $\alpha = \alpha^{\Delta} \circ \alpha^* = \overline{\alpha}^* \circ \overline{\alpha}^{\Delta}$ where $\overline{\alpha}^{\Delta} \colon P_* \rightarrow Q \sqcup \alpha^{-1}(*)$ so that $\overline{\alpha}^{\Delta}(i) = \alpha(i)$ if $\alpha(i) \ne *$ and $i$ otherwise.  Then the diagram
    \begin{equation}
    \begin{tikzcd}
        \A^Q \arrow{r}{\bi_{\alpha^\Delta}} \arrow{d}{\bi_{\overline{\alpha}^*}} & \A^{P \setminus \alpha^{-1}(*)} \arrow{d}{\bi_{\alpha^*}} \\
        \A^{Q \sqcup \alpha^{-1}(*) \setminus \{ * \}} \arrow{r}{\bi_{\overline{\alpha}^{\Delta}}}& \A^P
    \end{tikzcd}
    \end{equation}
    is Cartesian.  As a result, using $\Psi$-goodness, the diagram
    \begin{equation}
    \begin{tikzcd}
    \bi_{\alpha^{\Delta}}^* (\Psi)_{F(\alpha^*)} K \arrow{r} \arrow{d}& (\Psi)_{F(\overline{\alpha}^*)} \bi_{\overline{\alpha}^{\Delta}}^* K \arrow{d} \\
    (\Psi)_{F(\alpha)} K \arrow{r}{\sim}& (\Psi)_{F(\overline{\alpha}^*)} (\Psi)_{F(\overline{\alpha}^{\Delta})} K
    \end{tikzcd}
    \end{equation}
    commutes where the top horizontal arrow is an isomorphism by $\Psi$-goodness, the bottom horizontal arrow is an isomorphism by $\Psi$-factorizability, and right vertical arrow is an isomorphism by the local acyclicity over $\G_m^P$.  After taking invariants under $W^{\alpha^{-1}(*) \setminus \{ * \}}$ and applying Corollary~\ref{cor: unipotent general bases}, we conclude that Equation~\ref{eqn: local acyclic arrow} is an isomorphism.
\end{proof}

We now note that the factorization of $\alpha$ into $\alpha^{\Delta} \circ \alpha^*$ is compatible with composition of maps of finite pointed sets, so that $\bi_{(\beta \alpha)^{\Delta}} = \bi_{\alpha^{\Delta}} \circ \bi_{\beta^{\Delta}}$.  The previous corollary gives an isomorphism
\[ \bi_{(\beta \alpha)^{\Delta}}^* (\Psi_f^{u})_{F((\beta \alpha)^*)} K \cong \Upsilon^{\beta \alpha}_f(K). \]
For a composition, Achar and Riche construct a map
\begin{equation}
    \Upsilon^{\beta \alpha}_f(K) \rightarrow \Upsilon^{\beta}_{f^{\alpha}} \Upsilon^{\alpha}_f(K).
\end{equation}
On the other hand, if $(f, K)$ is $\Psi$-factorizable, we get an isomorphism
\begin{equation}\label{eqn: factorizable arrow}
    \bi_{(\beta \alpha)^{\Delta}}^* (\Psi_f)_{F((\beta \alpha)^*)} K \cong \bi_{\beta^{\Delta}}^* \bi_{\alpha^{\Delta}}^* (\Psi_f)_{F(\beta')} (\Psi_f)_{F(\alpha^*)} K,
\end{equation}
where $\beta' \colon (P \setminus \alpha^{-1}(*))_* \rightarrow (P \setminus (\beta \alpha)^{-1}(*))_*$ is the map sending all $i$ to itself except for $i$ such that $\alpha(i) \ne *$ but $\beta \alpha(i) = *$.  There is a canonical base change map for nearby cycles giving a map
\begin{equation}
    \bi_{\beta^{\Delta}}^* \bi_{\alpha^{\Delta}}^* (\Psi_f)_{F(\beta')} (\Psi_f^{u})_{F(\alpha^*)} K \rightarrow \bi_{\beta^{\Delta}}^* (\Psi_{f^{\alpha}})_{F(\beta')} \bi_{\alpha^{\Delta}}^* (\Psi_f^{u})_{F(\alpha^*)} K,
\end{equation}
where $f^{\alpha}$ is the pullback of $f$ along the map $\bi_{\alpha^{\Delta}}$.  There is an isomorphism
\begin{equation}
    \bi_{\beta^{\Delta}}^* (\Psi_{f^{\alpha}})_{F(\beta')} \bi_{\alpha^{\Delta}}^* (\Psi_f^{u})_{F(\alpha^*)} K \cong \Upsilon^{\beta}_{f^{\alpha}} \Upsilon^{\alpha}_{f}(K),
\end{equation}
coming from two applications of Corollary~\ref{cor: unipotent general bases}.  The following proposition shows that the above maps are compatible with each other.

\begin{proposition}\label{prop: unipotent compatibility}
    Let $(f,K)$ and $(f^{\alpha}, \bi_{\alpha^{\Delta}}^* (\Psi_f^u)_{F(\alpha^*)} K)$ be $\Psi$-factorizable and universally locally acyclic over $\G_m^P$ and $\G_m^Q$, respectively.  Let $\alpha \colon P_* \to Q_*$ and $\beta \colon Q_* \to R_*$ be maps in $\FinS_*$.  Then $\Upsilon^{\alpha}_f(K)$ is $\Psi$-good, and the following diagram commutes where the vertical arrows are isomorphisms coming from Corollary~\ref{cor: unipotent general bases}:
    \begin{equation}\label{eqn: compatibility square}
        \begin{tikzcd}
            \Upsilon^{\beta \alpha}_f(K) \arrow{r} \arrow{d}& \Upsilon^{\beta}_{f^{\alpha}} \Upsilon^{\alpha}_f(K) \arrow{d} \\
            \bi_{(\beta \alpha)^{\Delta}}^* (\Psi_f^u)_{F((\beta \alpha)^*)} K \arrow{r}& \bi_{\beta^{\Delta}}^* (\Psi_{f^{\alpha}}^u)_{F(\beta^*)} \bi_{\alpha^{\Delta}}^* (\Psi_f^u)_{F(\alpha^*)} K.
        \end{tikzcd}
    \end{equation}
    Moreover, the bottom row is an isomorphism, so the map $\Upsilon^{\beta \alpha}_f(K) \rightarrow \Upsilon^{\beta}_{f^{\alpha}} \Upsilon^{\alpha}_f(K)$ is an isomorphism.
\end{proposition}

\begin{proof}
    We explain how to make sense of the vertical arrows in Diagram~\ref{eqn: compatibility square}.  Since $(f, K)$ is $\Psi$-factorizable, it follows that $\Upsilon^{\alpha}_f(K) \cong \bi_{\alpha^{\Delta}}^* (\Psi_f^u)_{F(\alpha^*)} K$.  Since the pair $(f^{\alpha}, (\Psi_f)_{F(\alpha)} K)$ is $\Psi$-factorizable, the right vertical arrow is an isomorphism.  The left vertical arrow is also an isomorphism by the $\Psi$-factorizability.

    We may choose $\bc \colon P \rightarrow \Z_{\ge 1}$ sufficiently large so that
    \[ \Upsilon^{\beta \alpha}_f(K) \cong \bi_{\beta \alpha}^* j_* (K \otimes f^* \mathscr{L}_{\bc}) \]
    and $\bc$ is $\beta \alpha$-special, i.e. $\bc(p) = 1$ for $\beta \alpha(p) \ne *$.  Defining
    \[
    \begin{aligned}
    \ba(p) = \begin{cases}
        \bc(p) & \alpha(p) = * \\ 1 & \alpha(p) \ne *,
    \end{cases} &
    \bb(p) = \begin{cases}
        \bc(p) & \alpha(p) \ne * \\ 1 & \alpha(p) = *,
    \end{cases}
    \end{aligned}
    \]
    and
    \[ \bb'(q) = -|\alpha^{-1}(q)| + 1 + \sum_{p \in \alpha^{-1}(q)} \bb(p), \]
    the construction of Achar and Riche produces a map
    \begin{equation}
        \begin{tikzcd}
            \bi_{\beta \alpha}^* j_*(K \otimes \mathscr{L}_{\bc}) \arrow{r} \arrow{d}{\sim}& \bi_{\beta}^* j^{\alpha}_*(\bi_{\alpha}^* j_*(K \otimes \mathscr{L}_{\ba}) \otimes \mathscr{L}_{\bb'}) \arrow{d}{\sim} \\
            \Upsilon^{\beta \alpha}_f(K) \arrow{r}& \Upsilon^{\beta}_{f^{\alpha}} \Upsilon^{\alpha}_f(K).
        \end{tikzcd}
    \end{equation}
    Here $j^{\alpha}$ is the map $X \times_{\A^P} (\G_m)^Q \rightarrow X \times_{\A^P} \A^Q$.  The construction of the map in Equation~\ref{eqn: factorizable arrow} amounts to an adjunction, which yields a commutative diagram in conjunction with Corollary~\ref{cor: unipotent general bases},
    \begin{equation}
        \begin{tikzcd}
            \bi_{\beta \alpha}^* j_*(K \otimes \mathscr{L}_{\ba} \otimes \mathscr{L}_{\bb}) \arrow{r} \arrow{d}& \bi_{\beta \alpha}^* j_* (\bi_{\alpha^*})_{*} \bi_{\alpha^*}^* j_*(K \otimes \mathscr{L}_{\ba} \otimes \mathscr{L}_{\bb}) \arrow{d} \\
            \bi_{(\beta \alpha)^{\Delta}}^* (\Psi^u_{f}) K \arrow{r} &  \bi_{\beta^{\Delta}}^* \bi_{\alpha^{\Delta}}^* (\Psi_f^{u})_{F(\beta')} (\Psi_f^{u})_{F(\alpha^*)} K
        \end{tikzcd}
    \end{equation}
    with $\bc$ chosen large enough so that the vertical arrows are isomorphisms.  Since $j^{\alpha}_* = \bi_{\alpha^*}^* j_* (\bi_{\alpha^*})_{*}$, there is an isomorphism
    \[ \bi_{\beta \alpha}^* j_* (\bi_{\alpha^*})_{*} \bi_{\alpha^*}^* j_*(K \otimes \mathscr{L}_{\ba} \otimes \mathscr{L}_{\bb}) \cong \bi_{(\beta \alpha)^{\Delta}}^* \bi_{\beta'}^* j^{\alpha}_* \bi_{\alpha^*}^* j_*(K \otimes \mathscr{L}_{\ba} \otimes \mathscr{L}_{\bb}). \]
    Since $\mathscr{L}_{\bc} \cong \mathscr{L}_{\ba} \otimes \mathscr{L}_{\bb}$, we may form a commutative diagram
    \[
        \begin{tikzcd}
            \bi_{\beta \alpha}^* j_* (\bi_{\alpha^*})_{*} \bi_{\alpha^*}^* j_*(K \otimes \mathscr{L}_{\bc}) \arrow{r} \arrow{d} & \bi_{\beta} j^{\alpha}_*(\bi_{\alpha}^* j_*(K \otimes \mathscr{L}_{\ba}) \otimes \mathscr{L}_{bb'}) \arrow{d} \\
            \bi_{\beta^{\Delta}}^* \bi_{\alpha^{\Delta}}^* (\Psi_f^{u})_{F(\beta')} (\Psi_f^{u})_{F(\alpha^*)} K \arrow{r}& \bi_{\beta^{\Delta}}^* (\Psi_{f^{\alpha}}^{u})_{F(\beta')} \bi_{\alpha^{\Delta}}^* (\Psi_f^{u})_{F(\alpha^*)} K.
        \end{tikzcd}
    \]
    A diagram chase shows that the composition of the arrows in the above diagrams lines up with the construction of \cite[Section~2.10]{achar2023higher}, thus proving the relevant commutativity.
\end{proof}

In the case of the Beilinson-Drinfeld Grassmannian, we assert that the hypotheses in the previous theorem that use Conjecture~\ref{prop: factorizable} are already known in the product situation and can be transported along proper pushforward and smooth pullback, as in the construction of Satake sheaves.  Thus, we may state the following corollary without relying explicitly on Conjecture~\ref{prop: factorizable}.

\begin{corollary}\label{cor: maps agree}
    The isomorphisms in Proposition~\ref{prop: psi factorizability grassmannian} and \cite[Theorem~3.2]{achar2023higher} agree.
\end{corollary}

\begin{proof}
    The Satake sheaves $\mathcal{S}_{V_1 \boxtimes \dots \boxtimes V_{|I|}}$ on a Beilinson-Drinfeld Grassmannian over $\A^I$, with parahoric ramification along a point $0 \in \A^1$, satisfy the assumptions of Proposition~\ref{prop: unipotent compatibility}, so all the arrows in Proposition~\ref{prop: unipotent compatibility} are isomorphisms.  The universal local acyclicity property is a theorem of Richarz \cite{richarz2014new}.  Gaitsgory's theorem that the local Galois group acts unipotently on central sheaves \cite[Proposition~7]{gaitsgory1999construction} allows us to write $(\Psi_f^u)_{F(\alpha)} K$ as the full nearby cycle.  Gaitsgory's work on central sheaves allows us to write the object $\bi^*_{\alpha^{\Delta}} (\Psi_f)_{F(\alpha)} K$ as $\mathcal{S}_{W} \boxtimes Z_{V^{\otimes \alpha^{-1}(*) \cap P}}$ on the special fiber of the Beilinson-Drinfeld Grassmannian $\Gr_{G,I} \times_{A^P} (\G_m^Q \times \overline{0})$, where $W_q = \bigotimes_{i \in \alpha^{-1}(q)} V_i$ and $V^{\otimes \alpha^{-1}(*) \cap P} = V_{p_1} \otimes \dots \otimes V_{p_n}$ where $\{ p_1, \dots, p_n \} = \alpha^{-1}(*) \cap P$, considered as a constant sheaf (with monodromy automorphisms).  When we shriek extend to the Beilinson-Drinfeld Grassmannian $\Gr_{G,I} \times_{A^P} \A^Q$, the proof of the $\Psi$-factorizability from \cite[Proposition~3.14]{salmon2023restricted} generalizes to handle this case as well.
\end{proof}

\printbibliography

\end{document}